\documentclass[12pt]{amsart}
\calclayout

\usepackage{amsmath}
\usepackage{amssymb}
\usepackage{amsfonts}
\usepackage{amsthm}
\usepackage{mathrsfs}
\usepackage[all]{xy}
\usepackage{enumerate}
\usepackage{graphicx}
\usepackage{romannum}

\DeclareMathOperator{\spn}{span}

\def\RR{\mathbb{R}} 
\def\CC{\mathbb{C}} 
\def\NN{\mathbb{N}} 
 
\newtheorem{thm}{Theorem}[section]
\newtheorem{lem}[thm]{Lemma}
\newtheorem{prop}[thm]{Proposition}

\theoremstyle{definition}
\newtheorem{defn}[thm]{Definition}

\theoremstyle{remark}
\newtheorem{rmk}[thm]{Remark}

\numberwithin{equation}{section}

\def\BigRoman{\uppercase\expandafter{\romannumeral\number\count 
		255 }}
\def\Romannumeral{\afterassignment\BigRoman\count255=}

\begin{document}
	
	\title{Invariance of Diederich-Fornaess index}         
	\author{Jihun Yum}        
	\address{Department of Mathematics, Pohang University of Science and Technology, Pohang, 790-784, Republic of Korea}
	\email{wadragon@postech.ac.kr}
	\thanks{This research was supported by the SRC-GAIA (NRF-2011-0030044) through the National Research Foundation of Korea (NRF) funded by the Ministry of Education.}

	\begin{abstract}
		We show that the Diederich-Fornaess index of a domain in a Stein manifold is invariant under CR-diffeomorphisms. For this purpose we also improve CR-extension theorem.  
	\end{abstract}

	\maketitle

	\section{\bf Introduction}
	
	Let $D$ be a relatively compact domain with $C^k$ $(k \ge 1)$ smooth boundary in a complex manifold $M$. A $C^1$ function $\rho$ defined on a neighborhood $U$ of $\overline{D}$ is called a (global) {\it defining function} of $D$ if $D = \{z \in U : \rho(z) < 0 \}$ and $d\rho(z) \neq 0$ for any $z \in \partial D$. In 1977, K. Diederich and J. E. Fornaess (\cite{DieForn1977}) showed that, if $D$ is a pseudoconvex domain with $C^2$ boundary and $M$ is a Stein manifold, there exist a positive constant $\eta$ with $0 < \eta < 1$ and a defining function $\rho$ such that $-(-\rho)^{\eta}$ is strictly plurisubharmonic on $D$. The supremum of all such constant $\eta$ is called the {\it Diederich-Fornaess exponent} of $\rho$, denoted by $\eta_{\rho}(D)$; if no such $\eta$ exists, we define $\eta_{\rho}(D) = 0$. The supremum of all Diederich-Fornaess exponents is called the {\it Diederich-Fornaess index} of $D$, denoted by $\eta(D)$. Note that a generalization to a bounded pseudoconvex domain in $\CC^n$ with Lipschitz boundary was studied by Harrington (\cite{Harrington2007}).
	
	If $D$ is a strongly pseudoconvex domain, then there is a strictly plurisubharmonic defining function. Consequently, $\eta(D) = 1$. But the converse is not true; Fornaess and Herbig (\cite{ForHer2007}, \cite{ForHer2008}) showed that, if a $C^{\infty}$ smooth relatively compact domain $D \subset \CC^n$ admits a plurisubharmonic defining function, then $\eta(D) = 1$. In particular, the Thullen domain $\{(z,w) \in \CC^2 :  |z|^2 + |w|^4 < 1 \}$ has the Diederich-Fornaess index one, but it is clearly not strongly pseudoconvex. In fact, much more is known. For any sufficiently small $\eta > 0$ there exists a worm domain with Diederich-Fornaess index less than $\eta$ (\cite{DieForn1977-2}). Recently, Krantz, Liu, and Peloso (\cite{KranLiuPelo2016}) showed that for a bounded pseudoconvex domain $D \subset \CC^2$ with $C^{\infty}$ smooth boundary, the Diederich-Fornaess index is one if the Levi-flat sets form a real curve transversal to the holomorphic tangent vector fields on $\partial D$.
	
	Many significant far-reaching conclusions follow from the condition $\eta(D) > 0$. To cite only a few, we refer the reader to \cite{BernChar2000}, \cite{Blocki2005}, \cite{ChenFu2011}, \cite{DieForn1977}, \cite{Harrington2011}, \cite{Kohn1999}, \cite{PinZam2014}, and others.\\

	On the other hand, the following natural question has been posed. We express our gratitude to M. Adachi as well as S. Yoo., for pointing out this question to us. \\

	\begin{quote}
	{\bf Question:} {\it Is the Diederich-Fornaess index a biholomorphic or a CR invariant?}
	\end{quote}
	
	\vspace{3 mm}
	
	\noindent
	Indeed, the main theorem of this article is as follows:

	\begin{thm} \label{1.1}
	Let $M_1$, $M_2$ be Stein manifolds with dimension $n$ $(n \ge 2)$. Let $D_1 \subset M_1$, $D_2 \subset M_2$ be relatively compact domains with connected $C^k (k \ge 1)$ smooth boundaries. If there exists a $C^k$ smooth CR-diffeomorphism $f : \partial D_1 \rightarrow \partial D_2$, then the Diederich-Fornaess indices of $D_1$ and $D_2$ are equal.  
	\end{thm}

	There is a preceding result by M. Adachi (\cite{Adachi2015}) that, for a relatively compact domain in a complex manifold with $C^{\infty}$ smooth Levi-flat boundary, the Diederich-Fornaess index is invariant under CR-diffeomorphisms. \\

	Our proof of Theorem \ref{1.1} requires the following modification of the  theorem. \\

	\begin{thm} \label{1.2}
		Let $M_1$, $M_2$ be Stein manifolds with dimension $n$ $(n \ge 2)$. Let $D_1 \subset M_1$, $D_2 \subset M_2$ be relatively compact domains with connected $C^k (k \ge 1)$ smooth boundaries. If there exists $C^s (1 \le s \le k)$ smooth CR-diffeomorphism $f : \partial D_1 \rightarrow \partial D_2$, then there exists the unique $C^s$ smooth extension $F : \overline{D_1} \rightarrow \overline{D_2}$ such that $F|_{\partial D_1} = f$ and $F|_{D_1} : D_1 \rightarrow D_2$ is a biholomorphism.
	\end{thm}

	\section{\bf Improving CR-extension theorem} \label{2}

	Let $D$ be a domain in a complex manifold $M$, and $T_p(\partial D)$, $T_pM$ be the {\it real tangent space} of $\partial D$, $M$ at $p \in \partial D$, respectively. Then the {\it complexified tangent space of $T_pM$}, $\CC T_pM := \CC \otimes T_pM$, can be decomposed into the {\it holomorphic tangent space} $T^{(1,0)}_pM$ and {\it the anti-holomorphic tangent space} $T^{(0,1)}_pM$, i.e. $\CC T_pM = T^{(1,0)}_pM \oplus T^{(0,1)}_pM$. Let $T^{(1,0)}_p(\partial D) := \CC T_p(\partial D) \cap T^{(1,0)}_pM$ and  $T^{(0,1)}_p(\partial D) := \CC T_p(\partial D) \cap T^{(0,1)}_pM$.

	\begin{defn}
		A function $f \in C^1(\partial D)$ is called {\it a CR-function} on $\partial D$ if for all $p \in \partial D$ one has
		$$ L(f) = 0  \hspace{5 mm} \text{ for all } L \in T^{(0,1)}_p(\partial D). $$
	\end{defn}
	
	\begin{defn}
		Let $D_1$, $D_2$ be domains in complex manifolds $M_1$, $M_2$, respectively. A $C^1$ smooth map $F : \partial D_1 \rightarrow \partial D_2$ is called a {\it CR-map} if 
		$$ F(T^{(0,1)}_p(\partial D_1)) \subset T^{(0,1)}_p(\partial D_2) \hspace{5 mm} \text{ for all } p \in \partial D_1$$
		A $C^1$ smooth map $F : \partial D_1 \rightarrow \partial D_2$ is called a {\it CR-diffeomorphism} if it is a CR-map and a diffeomorphism.  
	\end{defn}
	
	\begin{defn}
		Let $H^{p,q}(M)$ be the $(p,q)$-th {\it Dolbeault cohomology of $M$} defined by
		$$ H^{p,q}(M) := \frac{\{\bar{\partial} \text{-closed (p,q)-form}\}}{\{\bar{\partial} \text{-exact (p,q)-form}\}} $$
		Let $H^{p,q}_c(M)$ be the $(p,q)$-th {\it Dolbeault cohomology of $M$ with compact support} defined by
		$$ H^{p,q}_c(M) := \frac{\{\bar{\partial} \text{-closed (p,q)-form with a compact support}\}}{\{\bar{\partial} \text{-exact (p,q)-form with a compact support}\}} $$
	\end{defn}
	
	\vspace{3 mm}

	It is well known that for a bounded domain $D$ in $\CC^n (n \ge 2)$ with connected boundary, every CR-function $f : \partial D \rightarrow \CC$ extends to $F : \overline{D} \rightarrow \CC$ holomorphically (\cite{Range1986}). We modify it for a relatively compact domain in a Stein manifold.
	
	\begin{rmk}
		The theorem above is widely known as Bochner-Hartogs theorem; but it was actually proven by Severi and Fichera (\cite{Fichera1957}, \cite{Severi1931}; see also \cite{Range2002}). We are indebted to M. Range for pointing this out. 
	\end{rmk}

	\begin{prop} [\cite{LaurThie2011}] \label{2.3}
		Let $M$ be a non-compact complex manifold such that $H^{0,1}_c(M) = 0$. For any relatively compact domain $D \subset M$ with $C^k (k \ge 1)$ smooth boundary such that $M \setminus \overline{D}$ is connected and any CR function $f$ of class $C^s (1 \le s \le k)$ on $\partial D$, there is a $C^s$ smooth function $F$ on $\overline{D}$, holomorphic on $D$, such that $F|_{\partial D} = f$.
	\end{prop}
	
	\begin{lem} \label{2.4}
		Let $M$ be a connected smooth real manifold with dimension $n$ $(n \ge 1)$, and $D \subset M$ be a domain with $C^k (k \ge 1)$ smooth boundary $\partial D$. If $\partial D$ is connected, then $M \setminus \overline{D}$ is connected. 
	\end{lem}
	
	\begin{proof}
		Suppose that $M \setminus \overline{D}$ is disconnected. Then we may write $M \setminus \overline{D}$ as the disjoint union of two open sets $X, Y \subset M$. Then it is easy to see that $\partial(M \setminus \overline{D}) = \partial X \cup \partial Y$ and $\partial(M \setminus \overline{D}) = \partial D$. Since $\partial D = \partial X \cup \partial Y$ is connected by the assumption, $\partial X \cap \partial Y \neq \emptyset$. Let $p \in \partial X \cap \partial Y$. Since $\partial D$ is the $C^k$ smooth embedded submanifold of $M$, there is a smooth chart $\psi : U_p \rightarrow \RR^n$ at $p$ such that 
		\begin{itemize}
			\item $\psi(p) = 0$,
			\item $\psi(U_p \cap D) = \psi(U_p) \cap \{(x_1, \cdots, x_n) \in \RR^n : x_n < 0 \}$,
			\item $\psi(U_p \cap \partial D) = \psi(U_p) \cap \{(x_1, \cdots, x_n) \in \RR^n : x_n = 0 \}$,
		\end{itemize}
		where $U_p$ is an open neighborhood of $p$ in $M$. Then $\psi(U_p \cap (M \setminus \overline{D})) = \psi(U_p) \cap \{(x_1, \cdots, x_n) \in \RR^n : x_n > 0 \}$ implies that $X \cap Y \neq \emptyset$, which is a contradiction. 
	\end{proof}

	\begin{thm}[CR-extension theorem] \label{2.5}
		Let $M$ be a Stein manifold with dimension $n$ $(n \ge 2)$. Let $D \subset M$ be a relatively compact domain with connected $C^k (k \ge 1)$ smooth boundary. If there exists a $C^s (1 \le s \le k)$ smooth CR-function $f : \partial D \rightarrow \CC$, then there is a $C^s$ smooth function $F$ on $\overline{D}$, holomorphic on $D$, such that $F|_{\partial D} = f$.
	\end{thm}
	
	\begin{proof}
		Since $M$ is a Stein manifold, $M$ is non-compact. By Serre Duality Theorem (\cite{ChakShaw2012}), $H^{0,1}_c(M)$ is dual to $H^{n,n-1}(M)$ and $H^{n,n-1}(M) = 0$ since $M$ is Stein (\cite{Hormander1990}). Therefore $H^{0,1}_c(M) = 0$. By Lemma \ref{2.4}, the connectedness of the boundary $\partial D$ implies the connectedness of $M \setminus \overline{D}$. Hence the theorem follows by Proposition \ref{2.3}.
	\end{proof}

	We are now ready to prove Theorem \ref{1.2}. \\

	\subsection{Proof of Theorem 1.2.}
		
	Since $M_2$ is a Stein manifold, there is an embedding of $M_2$ into $\CC^N$ for some large $N \in \NN$, and hence from now on we may regard $M_2$ as an embedded complex  submanifold of $\CC^N$. Therefore, we may write the given CR-diffeomorphism $f$ as $f = (f_1, \cdots, f_N )$, where each $f_i : \partial D_1 \rightarrow \CC$ ($i = 1, \cdots, N$) is a CR-function. By CR-extension theorem(Theorem \ref{2.5}), $f$ extends to $F : \overline{D_1} \rightarrow \CC^N$ holomorphically. \\
		
	First, we show that the image of $F$ is in $M_2$ (i.e. $F(\overline{D_1}) \subset M_2$). Let $p \in \partial D_1, q \in \partial D_2$ with $f(p)=q$. Since $M_2$ is a Stein manifold, there exists a holomorphic map $\phi_q : M_2 \rightarrow \CC^n$ such that $\phi_q$ is a local biholomorphism at $q$. Define $g_q : \partial D_1 \rightarrow \CC^n$ by $g_q := \phi_q \circ f$. By CR-extension theorem(Theorem \ref{2.5}), $g_q$ extends to $G_q : \overline{D_1} \rightarrow \CC^n$ holomorphically. Let $U_p$ be an open neighborhood of $p$ in $M_1$ such that $\phi_q^{-1}$ is well-defined on $G_q(U_p \cap \overline{D_1})$. Define $H_{p=f^{-1}(q)} : U_p \cap \overline{D_1} \rightarrow M_2$ by $H_p := \phi_q^{-1} \circ G_q$.	Patching $H_p$ for each $p \in \partial D_1$, we may construct $H : U \cap \overline{D_1} \rightarrow M_2$, holomorphic on $U \cap D_1$, for some neighborhood $U$ of $\partial D_1$. Since the holomorphic extension of CR-function is unique, $H$ is well-defined and $F|_{U \cap \overline{D1}} = H$. Therefore $F(U \cap \overline{D_1}) \subset M_2$. \\
	
	Now let $\tilde{p} \in U \cap D_1$, $p \in D_1$ and $\gamma$ be a curve in $D_1$ joining $\tilde{p}$ and $p$. We choose $m \in \NN$ and $p_j, q_j, U_j, W_j, V_j$ $(j=1, \cdots, m)$ as follows. 
		
		\begin{itemize}
			\item $p_1 := \tilde{p}$, $p_m := p$, and $p_j \in \gamma$ for every $j$.
			\item $\{U_j\}^{m}_{j=1}$ is an open cover of $\gamma$ with $p_j \in U_j$, and $U_j \cap U_{j+1} \neq \emptyset$ for every $j$.
			\item $q_j = F(p_j)$ for every $j$.
			\item $W_j$ is a neighborhood of $q_j$ in $\CC^N$ such that $F(U_j) \subset W_j$ and $V_j := W_j \cap M_2$ is the zero set of holomorphic functions $\psi^{j}_{\alpha} : W_j \rightarrow \CC$ $(\alpha = 1, \cdots, N-n)$ whenever $V_j$ is non-empty for every $j$.
		\end{itemize}
		
		\noindent
		Since $(U \cap D_1) \cap U_1 \neq \emptyset$ and $(\psi^{1}_{\alpha} \circ F)((U \cap D_1) \cap U_1) \equiv 0$, $(\psi^{1}_{\alpha} \circ F)(U_1) \equiv 0$ for each $\alpha = 1, \cdots, N-n$, (i.e. $F(U_1) \subset M_2$). 
		Similarly, since $(\psi^{2}_{\alpha} \circ F)(U_1 \cap U_2) \equiv 0$, $(\psi^{2}_{\alpha} \circ F)(U_2) \equiv 0$ for each $\alpha = 1, \cdots, N-n$, (i.e. $F(U_2) \subset M_2$). By induction, it follows that $F(U_m) \subset M_2$, which means $F(p) \in M_2$. Since $p \in D_1$ is arbitrary, $F(\overline{D_1}) \subset M_2$. \\

		\vspace{5 mm}
		\includegraphics[scale = 0.32]{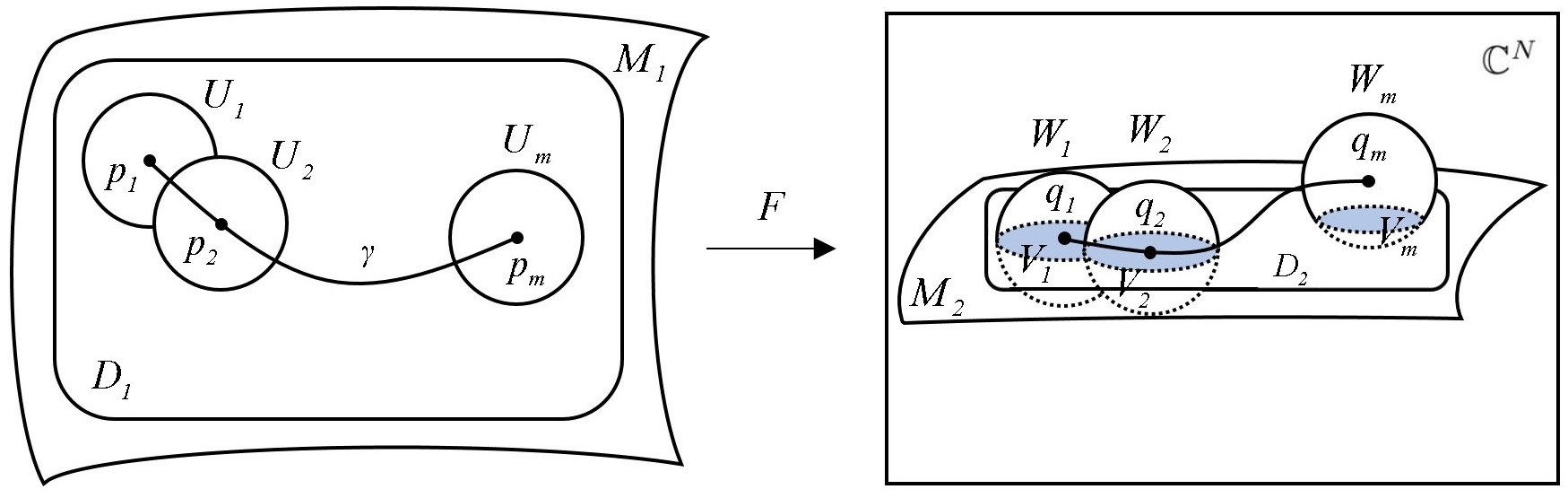}
		$$ \text{\bf Fig. 1.} $$
		\vspace{2 mm}

		Next, we show that $F(D_1) = D_2$. Take a $C^k$ smooth extension of $F$ on an open neighborhood of $\overline
		{D_1}$, and denote it again by $F$. Then we proceed with the following three steps: \\ 
		
		Step1) $det(J_{\CC} F) \neq 0$ on $\partial D_1$ : Choose any point $p \in \partial D_1$ and let $q := F(p) \in \partial D_2$. Let $\{x_1,y_1, \cdots, x_n, y_n\}$ be a real coordinate system at $p$ such that $\big\{\frac{\partial}{\partial x_1}\big|_p, \frac{\partial}{\partial y_1}\big|_p, \cdots, \frac{\partial}{\partial x_n}\big|_p\big\}$ spans $T_p(\partial D_1)$. Let $J_1$ and $J_2$ be the complex structures of $M_1$ and $M_2$, respectively. Since $F$ is holomorphic on $D_1$, $\bar{\partial}F = 0$ (i.e. $J_2 \circ dF = dF \circ J_1$) on $\overline{D_1}$. Consequently, $det(J_{\RR}F) = |det(J_{\CC}F)|^2$ holds at $p$. Therefore it is enough to show that $d_pF : T_p M_1 \rightarrow T_q M_2$ is $\RR$-linear isomorphism. If $d_pF\big(\frac{\partial}{\partial y_n}\big|_p\big) = 0$, then
		$$ d_pF\Big(-\frac{\partial}{\partial x_n}\Big|_p\Big) = d_pF\Big(J_1\Big(\frac{\partial}{\partial y_n}\Big|_p\Big)\Big) = J_2 \circ d_pF\Big(\frac{\partial}{\partial y_n}\Big|_p\Big) = 0, $$ 
		which contradicts that $F|_{\partial D_1}$ is a diffeomorphism. Suppose that $d_pF\Big(\frac{\partial}{\partial y_n}\Big|_p\Big) = w \neq 0 \in T_q(\partial D_2)$. Since $F|_{\partial D_1}$ is a diffeomorphism, there exists $v \in T_p(\partial D_1)$ such that $d_pF(v) = w$. Then 
		$$ d_pF\Big(\spn\Big\{\frac{\partial}{\partial x_n}\Big|_p, \frac{\partial}{\partial y_n}\Big|_p \Big\} \Big) =  d_pF(\spn\{v, J_1v\}) = \spn\{w, J_2w\} \subset T_q(\partial D_2), \text{ and}$$ 
		$$ \spn\Big\{\frac{\partial}{\partial x_n}\Big|_p, \frac{\partial}{\partial y_n}\Big|_p\Big\} \cap \spn{\{ v, J_1v \}} = \{0\} $$
		imply that there exists $X \in \spn{\{v, J_1v\}} \subset T_p(\partial D_1)$ such that  $d_pF(X) = d_pF\big(\frac{\partial}{\partial x_n}\big)$ with $X \neq \frac{\partial}{\partial x_n}$. This contradicts that $F|_{\partial D_1}$ is a diffeomorphism. Therefore $d_pF\big(\frac{\partial}{\partial y_n}\big|_p\big) \notin T_q(\partial D_2)$ and $d_pF : T_p M_1 \rightarrow T_q M_2$ is $\RR$-linear isomorphism.

		Step2) $det(J_{\CC} F) \neq 0$ on $D_1$ : Suppose that $Z := \{ z \in D_1 : det(J_{\CC} F|_{z}) = 0 \}$ is non-empty. Note that $Z$ is a well-defined closed $(n-1)$-dimensional analytic variety in $D_1$. Since $det(J_{\CC} F) \neq 0$ on $\partial D_1$ by (1), $Z$ is a compact analytic variety in $D_1$. Since $M_1$ is Stein, $Z$ should be finite (i.e. $\text{dim}_{\CC}(Z) = 0$). This contradicts that $n \ge 2$.  
		
		Step3) $F(D_1) = D_2$ : Since $det(J_{\CC} F) \neq 0$ on $D_1$ by (2), $F|_{D_1}$ is a local biholomorphism, hence an open map. Therefore $F(\partial D_1) = \partial(F(D_1))$ holds so $F(D_1)$ is either $D_2$ or $M_2 \setminus \overline{D_2}$. Since $M_2 \setminus \overline{D_2}$ is non-compact, $F(D_1) = D_2$. \\

		Now, we show that $F|_{D_1} : D_1 \rightarrow D_2$ is a biholomorphism. Since $f^{-1} : \partial D_2 \rightarrow \partial D_1$ is also a $C^k$ smooth CR-diffeomorphism, by using the same argument as above, there exists the holomorphic extension $G$ of $f^{-1}$ such that $G(D_2) = D_1$. Now consider the map $G \circ F : \overline{D_1} \rightarrow \overline{D_1}$. Since $G \circ F$ is holomorphic on $D_1$ and $G \circ F|_{\partial D_1} = id_{\partial D_1}$, $G \circ F = id_{\overline{D_1}}$. Similarly, by using the same argument, $F \circ G = id_{\overline{D_2}}$. Therefore, $F|_{D_1}$ is a biholomorphism.  \hfill $\Box$

	\section{\bf Invariance of Diederich-Fornaess index} \label{4}
	
	In Theorem \ref{1.2}, we extended the given CR-diffeomorphism $f : \partial D_1 \rightarrow \partial D_2$ to the interior of $D_1$ holomorphically. In this section, we extend it to the exterior of $D_1$ smoothly preserving the injectivity, and prove that the Diederich-Fornaess index is invariant under CR-diffeomorphisms.

	\begin{lem} \label{3.1}
		Let $M_1$, $M_2$ be Stein manifolds with dimension $n$ $(n \ge 2)$. Let $D_1 \subset M_1$, $D_2 \subset M_2$ be relatively compact domains with connected $C^k (k \ge 1)$ smooth boundaries. If there exists a $C^k$ smooth CR-diffeomorphism $f : \partial D_1 \rightarrow \partial D_2$, then there exist neighborhoods $U_1 \subset \overline{D_1}$, $U_2 \subset \overline{D_2}$, and a $C^k$ smooth diffeomorphism $F : U_1 \rightarrow U_2$ such that $F|_{\partial D_1} = f$ and $F|_{D_1} : D_1 \rightarrow D_2$ is a biholomorphism. 
	\end{lem}
	
	\begin{proof}
		By Theorem \ref{1.2}, there exists the $C^k$ smooth extension $F : \overline{D_1} \rightarrow \overline{D_2}$ such that $F|_{D_1} : D_1 \rightarrow D_2$ is a biholomorphism. Take a $C^k$ smooth extension of $F$ on an open neighborhood of $\overline{D_1}$, and denote it again by $F$. We show that there exist neighborhoods $U_1 \subset \overline{D_1}$, $U_2 \subset \overline{D_2}$ such that $F|_{U_1} : U_1 \rightarrow U_2$ is a diffeomorphism. In the proof of Theorem \ref{1.2}, we showed that $det(J_{\RR}F|_p) \neq 0$ for all $p \in \partial D_1$. Hence $F$ is a local diffeomorphism at every $p \in \partial D_1$. 
		
		Next, choose any Riemannian metric $g$ on $M_2$. Define $\text{dist}(w) : M_2 \rightarrow \RR$ by the distance from $w$ to $\partial D_2$ with respect to $g$. Let $V_{\delta} := \{ w \in M_2 : \text{dist}(w) < \delta \}$. Choose $\delta > 0 $ so that for each $w \in V_{\delta}$ there exists the unique closed point in $\partial D_2$ from $w$. Define $\psi(w) : V_{\delta} \rightarrow \partial D_2$ by the closed point in $\partial D_2$ from $w$. For each $q \in \partial D_2$, define $V_q := \{ w \in V_{\delta} : \psi(w) \in \partial D_2 \cap \mathbb{B}(q,\epsilon) \}$, where $\mathbb{B}(q,\epsilon)$ is the geodesic ball centered at $q$ with radius $\epsilon > 0$.
		Let $U_{F^{-1}(q)}$ be the connected component of $F^{-1}(V_q)$ containing $F^{-1}(q)$.
		By letting $\epsilon > 0, \delta > 0$ small, if necessarily, $F|_{U_p} : U_p \rightarrow V_{F(p)}$ is injective for all $p \in \partial D_1$. 
		Let $U := \underset{p \in \partial D_1}{\bigcup} U_p$ and $V := \underset{q \in \partial D_2}{\bigcup}  V_q$, then $U$ and $V$ are neighborhoods of $\partial D_1$ and $\partial D_2$, respectively.
		
		Now choose any $z_1, z_2 \in U \setminus \overline{D_1}$. Then $z_1 \in U_{p_1}$, $z_2 \in U_{p_2}$ for some $p_1, p_2 \in \partial D_1$. Let $q_1 := F(p_1), q_2 := F(p_2)$. Suppose that $w := F(z_1) = F(z_2) \in V_{q_1} \cap V_{q_2}$. If $p_1 = p_2$, then since $F|_{U_{p_1}} : U_{p_1} \rightarrow V_{q_1}$ is a diffeomorphism, $z_1 = z_2$. If $p_1 \neq p_2$, then the distance realizing geodesic $\gamma$ joining $w$ and $q_0 := \psi(w)$ is in $V_{q_1} \cap V_{q_2}$. Note that $(F|_{U_{p_1}})^{-1} (\gamma)$, $(F|_{U_{p_2}})^{-1} (\gamma)$ are curves joining $z_1, F^{-1}(q_0)$ and $z_2, F^{-1}(q_0)$, respectively. Therefore if $z_1 \neq z_2$, then $F|_{U_{p_1}} : U_{p_1} \rightarrow V_{q_1}$ is not injective near $F^{-1}(q_0)$, which is a contradiction. By letting $U_1 := U \cup D_1$, $U_2 := V \cup D_2$, Lemma \ref{3.1}  is proved.
	\end{proof}

	\vspace{5 mm}
	\includegraphics[scale = 0.52]{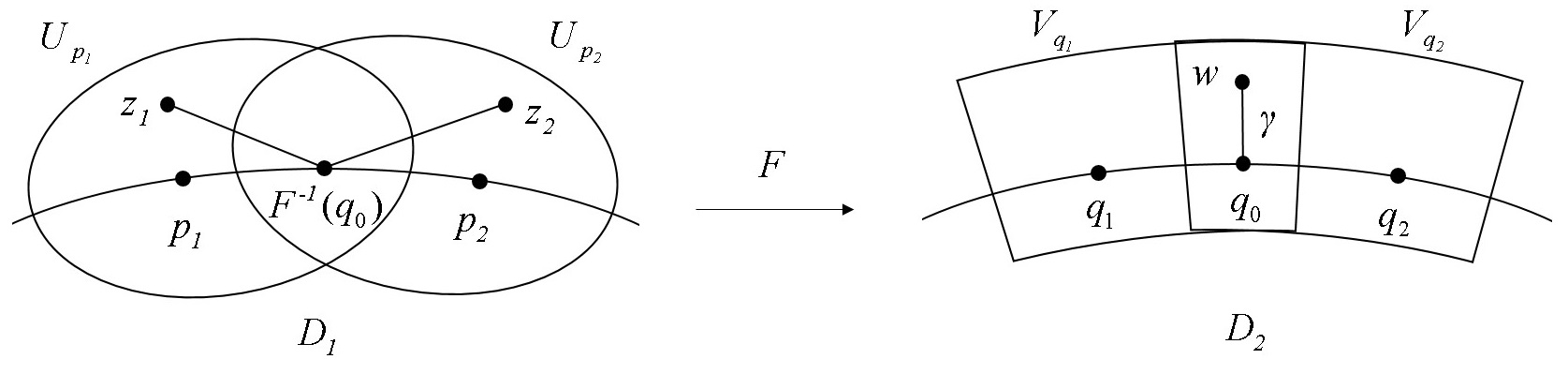}
	$$ \text{\bf Fig. 2.} $$
	\vspace{2 mm}

	\subsection{Proof of Theorem 1.1}

		By Lamma \ref{3.1}, there exists a $C^k$ smooth diffeomorphism $F : U_1 \rightarrow U_2$ such that $F|_{\partial D_1} = f$ and $F|_{D_1} : D_1 \rightarrow D_2$ is a biholomorphism, where $U_1$ and $U_2$ are neighborhoods of $\overline{D_1}$ and $\overline{D_2}$, respectively. If $\eta(D_1)$ and $\eta(D_2)$ are both zero, then we are done. So we may assume that $\eta(D_2) > 0$. Let $\rho_2$ be a $C^s$ $(1 \le s \le k)$ smooth defining function of $D_2$, then $\rho_1 := \rho_2 \circ F$ is a $C^s$ smooth defining function of $D_1$. If $-(-\rho_2)^{\eta}$ is strictly plurisubharmonic on $D_2$ for some $0 < \eta \le 1$, then $-(-(\rho_2 \circ F))^{\eta}$ is also strictly plurisubharmonic on $D_1$ with the same $\eta$ because $F$ is holomorphic on $D_1$. This implies that $\eta(D_1) \ge \eta$, and hence $\eta(D_1) \ge \eta(D_2)$. Therefore $\eta(D_1) > 0$. Applying the same argument as above with $F^{-1} : U_2 \rightarrow U_1$, one can see that $\eta(D_1) \le \eta(D_2)$. Hence $\eta(D_1) = \eta(D_2)$.  \hfill $\Box$ \\

	
	\textit{Acknowledgements}: The author would like to express his deep gratitude to Professor Kang-Tae Kim for valuable guidance and encouragements, and to Professor R. M. Range for valuable advice. This work is part of author's Ph.D. dissertation at Pohang University of Science and Technology.
	

	\bibliographystyle{amsplain}

\end{document}